\documentclass[a4, 12pt]{amsart}
\usepackage{amssymb}
\usepackage{amstext}
\usepackage{amsmath}
\usepackage{amscd}
\usepackage{amsthm}
\usepackage{amsfonts}
\usepackage{enumerate}
\usepackage{graphicx}
\usepackage{latexsym}
\theoremstyle{plain}
\newtheorem{thm}{Theorem}[section]
\newtheorem{prop}[thm]{Proposition}
\newtheorem{cor}[thm]{Corollary}
\newtheorem*{thm*}{Theorem}
\newtheorem*{cor*}{Corollary}
\theoremstyle{definition}
\newtheorem{dfn}[thm]{Definition}
\newtheorem{ex}[thm]{Example}

\newtheorem{ques}[thm]{Question}
\newtheorem*{conv1}{1}
\newtheorem*{conv2}{2}
\newtheorem*{conv3}{3}
\newtheorem{conj}[thm]{Conjecture}
\newtheorem*{conj*}{Conjecture}
\numberwithin{equation}{thm}
\def\Tor{\operatorname{Tor}}

\def\edim{\operatorname{edim}}
\def\e{\operatorname{e}}
\def\mod{\operatorname{\mathsf{mod}}}
\def\add{\operatorname{\mathsf{add}}}
\def\zero{\mathbf{0}}
\def\Image{\operatorname{Im}}
\def\ext{\operatorname{\mathsf{ext}}}
\def\filt{\operatorname{\mathsf{filt}}}
\def\m{\mathfrak m}
\def\a{\mathfrak a}
\def\X{{\mathcal X}}
\def\Y{{\mathcal Y}}
\def\c{{\mathcal C}}
\begin{document}
\setlength{\baselineskip}{15pt}
\title[A nontrivial extension-closed subcategory]{When is there a nontrivial extension-closed subcategory?}
\author{Ryo Takahashi}
\address{Department of Mathematical Sciences, Faculty of Science, Shinshu University, 3-1-1 Asahi, Matsumoto, Nagano 390-8621, Japan}
\email{takahasi@math.shinshu-u.ac.jp}
\thanks{2010 {\em Mathematics Subject Classification.} 13C60}
\thanks{{\em Key words and phrases.} extension-closed subcategory, stretched Artinian local ring, extension closure}
\begin{abstract}
Let $R$ be a commutative Noetherian local ring, and denote by $\mod R$ the category of finitely generated $R$-modules.
In this paper, we consider when $\mod R$ has a nontrivial extension-closed subcategory.
We prove that this is the case if there are part of a minimal system of generators $x,y$ of the maximal ideal with $xy=0$, and that it holds if $R$ is a stretched Artinian Gorenstein local ring which is not a hypersurface.
\end{abstract}
\maketitle

\section*{Introduction}

Let $R$ be a commutative Noetherian local ring with maximal ideal $\m$.
Denote by $\mod R$ the category of finitely generated $R$-modules.
An {\em extension-closed} subcategory of $\mod R$ is by definition a nonempty strict full subcategory of $\mod R$ closed under direct summands and extensions.
The zero $R$-module, the finitely generated free $R$-modules and all the finitely generated $R$-modules form extension-closed subcategories of $\mod R$, respectively.
We call these three subcategories {\em trivial} extension-closed subcategories of $\mod R$.

In this paper, we consider when there are only trivial extension-closed subcategories and when a nontrivial one exists.
In the case where $R$ is an Artinian hypersurface, all the extension-closed subcategories of $\mod R$ are trivial.
Our conjecture is that the converse also holds true.

\begin{conj*}
The following are equivalent.
\begin{enumerate}[(1)]
\item
$R$ is an Artinian hypersurface.
\item
$\mod R$ has only trivial extension-closed subcategories.
\end{enumerate}
\end{conj*}

Both conditions in this conjecture imply that $R$ is an Artinian Gorenstein local ring.
The conjecture holds if $R$ is a complete intersection.

The main result of this paper is the following theorem.

\begin{thm*}
Let $x,y$ be part of a minimal system of generators of $\m$ with $xy=0$.
Then $R/\m$ does not belong to the smallest extension-closed subcategory of $\mod R$ containing $R/(x)$, and hence it is a nontrivial extension-closed subcategory.
\end{thm*}

Let $R$ be an Artinian local ring of length $l$ with embedding dimension $e$.
Recall that $R$ is said to be {\em stretched} if $\m^{l-e}\ne 0$.
An Artinian Gorenstein local ring which is not a field and the cube of whose maximal ideal is zero is an example of a stretched Artinian Gorenstein local ring.
The above theorem yields the following corollary, which guarantees that our conjecture holds when $R$ is a stretched Artinian Gorenstein local ring.

\begin{cor*}
Let $R$ be a stretched Artinian Gorenstein local ring.
Then the following are equivalent.
\begin{enumerate}[\rm (1)]
\item
$R$ is an Artinian hypersurface.
\item
$\mod R$ has only trivial extension-closed subcategories.
\end{enumerate}
\end{cor*}

\section*{Convention}

\begin{conv1}
Throughout the rest of this paper, we assume that all rings are commutative Noetherian local rings, and that all modules are finitely generated.
Let $R$ be a commutative Noetherian local ring.
We denote by $\m$ the maximal ideal of $R$, by $k$ the residue field of $R$ and by $\mod R$ the category of finitely generated $R$-modules.
\end{conv1}

\begin{conv2}
Let $\c$ be a category.
In this paper, by a {\em subcategory} of $\c$, we always mean a nonempty strict full subcategory of $\c$.
(Recall that a subcategory $\X$ of $\c$ is said to be {\em strict} if every object of $\c$ that is isomorphic in $\c$ to some object of $\X$ belongs to $\X$.)
By the {\em subcategory} of $\c$ consisting of objects $\{M_\lambda\}_{\lambda\in\Lambda}$, we always mean the smallest strict full subcategory of $\c$ to which $M_\lambda$ belongs for all $\lambda\in\Lambda$.
Note that this coincides with the full subcategory of $\c$ consisting of all objects $X\in\c$ such that $X\cong M_\lambda$ for some $\lambda\in\Lambda$.
\end{conv2}

\begin{conv3}
We will often omit a letter indicating the base ring if there is no fear of confusion.
\end{conv3}

\section{Some observations}

We begin with recalling the precise definition of an extension-closed subcategory of $\mod R$.

\begin{dfn}
Let $\X$ be a subcategory of $\mod R$.
We say that $\X$ is {\em extension-closed} if $\X$ satisfies the following two conditions.
\begin{enumerate}[(1)]
\item
$\X$ is closed under direct summands:
if $M$ is an $R$-module in $\X$ and $N$ is a direct summand of $M$, then $N$ is also in $\X$.
\item
$\X$ is closed under extensions:
for every exact sequence $0 \to L \to M \to N \to 0$ of $R$-modules, if $L$ and $N$ are in $\X$, then $M$ is also in $\X$.
\end{enumerate}
\end{dfn}

For an $R$-module $X$, we denote by $\add_RX$ the {\em additive closure} of $X$, namely, the smallest subcategory of $\mod R$ containing $X$ which is closed under finite direct sums and direct summands.
This is nothing but the subcategory of $\mod R$ consisting of all direct summands of finite direct sums of copies of $X$.
Note that the additive closure $\add_RR$ of $R$ is the same as the subcategory of $\mod R$ consisting of all free $R$-modules.

We call the subcategory of $\mod R$ consisting of the zero $R$-module the {\em zero subcategory} of $\mod R$, and denote it by $\zero$.
Clearly,
$$
\zero,\ \add R,\ \mod R
$$
are all extension-closed subcategories of $\mod R$.
We call these three subcategories trivial extension-closed subcategories of $\mod R$.

\begin{dfn}
We say that $\mod R$ {\em has only trivial extension-closed subcategories} if all the extension-closed subcategories of $\mod R$ are $\zero$, $\add R$ and $\mod R$.
If there exists an extension-closed subcategory of $\mod R$ other than these three, then we say that $\mod R$ {\em has a nontrivial extension-closed subcategory}.
\end{dfn}

Over an Artinian hypersurface, there exists no nontrivial extension-closed subcategory.

\begin{prop}\label{mainp}
If $R$ is an Artinian hypersurface, then $\mod R$ has only trivial extension-closed subcategories.
\end{prop}

\begin{proof}
This is proved in \cite[Proposition 5.6]{stcm}.
For the convenience of the reader, we give here a proof.
There exist a discrete valuation ring $S$ with maximal ideal $(x)$ and a positive integer $n$ such that $R$ is isomorphic to $S/(x^n)$.
Applying to $S$ the structure theorem for finitely generated modules over a principal ideal domain, we have
$$
\mod R=\add_R(R\oplus R/(x)\oplus R/(x^2)\oplus\cdots\oplus R/(x^{n-1})).
$$
Let $\X$ be an extension-closed subcategory of $\mod R$.
Suppose that $\X$ is neither $\zero$ nor $\add R$.
Then $\X$ contains $R/(x^l)$ for some $1\le l\le n-1$.
For each integer $1\le i\le n-1$ there exists an exact sequence
$$
0 \to R/(x^i) \overset{f}{\to} R/(x^{i-1})\oplus R/(x^{i+1}) \overset{g}{\to} R/(x^i) \to 0
$$
of $R$-modules, where $x^0:=1$, $f(\overline a)=\binom{\overline a}{\overline{ax}}$ and $g(\binom{\overline a}{\overline b})=\overline{ax-b}$.
Hence $\X$ contains both $R/(x^{l-1})$ and $R/(x^{l+1})$.
An inductive argument implies that $\X$ contains $R/(x), R/(x^2), \dots, R/(x^{n-1}), R/(x^n)=R$.
Therefore $\X$ coincides with $\mod R$.
\end{proof}

We conjecture that the converse of Proposition \ref{mainp} also holds.
The main purpose of this paper is to study this conjecture.

\begin{conj}\label{mainc}
If $\mod R$ has only trivial extension-closed subcategories, then $R$ is an Artinian hypersurface.
\end{conj}

One can show that the assumption of Conjecture \ref{mainc} implies that $R$ is Artinian and Gorenstein.

\begin{prop}\label{artgor}
If $\mod R$ has only trivial extension-closed subcategories, then $R$ is an Artinian Gorenstein ring.
\end{prop}

\begin{proof}
First, let $\X$ be the subcategory of $\mod R$ consisting of all $R$-modules of finite length.
Clearly, $\X$ is an extension-closed subcategory of $\mod R$.
Using the fact that $\X$ contains $k$ and our assumption, we easily deduce that $\X$ coincides with $\mod R$, which implies that $R$ is Artinian.

Next, let $\Y$ be the subcategory of $\mod R$ consisting of all injective $R$-modules.
It is obvious that $\Y$ is extension-closed, and the injective hull of $k$ belongs to $\Y$.
Our assumption implies that $\Y$ is equal to $\add R$, and we see that $R$ is Gorenstein.
\end{proof}

In the proposition below, we give a sufficient condition for $\mod R$ to have a nontrivial extension-closed subcategory.
This sufficient condition is a little complicated, but by using this, we will obtain some explicit sufficient conditions.

\begin{prop}\label{srmn}
Let $S\to R$ be a homomorphism of local rings.
Assume that there exist $R$-modules $M,N$ such that: 
\begin{itemize}
\item
$M$ is $S$-flat and not $R$-free,
\item
$N$ is not $S$-flat.
\end{itemize}
Then $\mod R$ has a nontrivial extension-closed subcategory.
\end{prop}

\begin{proof}
Let $\X$ be the subcategory of $\mod R$ consisting of all $S$-flat $R$-modules.
It is easy to see that $\X$ is an extension-closed subcategory of $\mod R$.
The existence of $M$ and $N$ shows that $\X$ does not coincide with any of $\zero$, $\add R$, $\mod R$.
\end{proof}

The following result is a direct consequence of Proposition \ref{srmn}.

\begin{cor}\label{sri}
Suppose that there exist a local subring $S\subsetneq R$ which is not a field and an ideal $I\subsetneq R$ such that the composition $S\to R\to R/I$ is an isomorphism.
Then $\mod R$ has a nontrivial extension-closed subcategory.
\end{cor}

\begin{proof}
Apply Proposition \ref{srmn} to $M=R/I$ and $N=k$.
\end{proof}

The next three results, which give explicit sufficient conditions for $\mod R$ to have a nontrivial extension-closed subcategory, are all deduced from Corollary \ref{sri}.

\begin{cor}\label{idealize}
Let $S$ be a local ring which is not a field and $N$ a nonzero $S$-module.
Let $R=S\ltimes N$ be the idealization of $N$ over $S$.
Then $\mod R$ has a nontrivial extension-closed subcategory.
\end{cor}

\begin{proof}
Setting $I=\{\,(0,n)\in R\mid n\in N\,\}$, we see that the composite map $S\to R\to R/I$ of natural homomorphisms is an isomorphism.
Corollary \ref{sri} yields the conclusion.
\end{proof}

\begin{cor}\label{ctensor}
Let $S,T$ be complete local rings which are not fields and have the same coefficient field $k$.
Let $R=S\operatorname{\widehat{\otimes}}_kT$ be the complete tensor product of $S$ and $T$ over $k$.
Then $\mod R$ has a nontrivial extension-closed subcategory.
\end{cor}

\begin{proof}
We can write $S\cong k[[x_1,\dots,x_n]]/(f_1,\dots,f_a)$ and $T\cong k[[y_1,\dots,y_m]]/(g_1,\dots,g_b)$, where $n,m\ge 1$, $f_1,\dots,f_a\in(x_1,\dots,x_n)^2$ and $g_1,\dots,g_b\in(y_1,\dots,y_m)^2$.
Then $R$ is isomorphic to the ring $k[[x_1,\dots,x_n,y_1,\dots,y_m]]/(f_1,\dots,f_a,g_1,\dots,g_b)$.
The composition $S\to R\to R/(y_1,\dots,y_m)R$ of natural maps is an isomorphism, and we can use Corollary \ref{sri}.
\end{proof}

The following result is due to Shiro Goto.

\begin{cor}
Let $R=k[[X_1,\dots,X_n,Y]]/\a$ be a residue ring of a formal power series ring over a field $k$ with $n\ge 1$.
Assume that $Y^{l+1}\in\a\subseteq(X_1,\dots,X_n,Y)^{l+1}$ holds for some $l\ge 1$.
Then $\mod R$ has a nontrivial extension-closed subcategory.
\end{cor}

\begin{proof}
Let $x_1,\dots,x_n,y\in R$ be the residue classes of $X_1,\dots,X_n,Y$.
Let $k[[y]]$ be the $k$-subalgebra of $R$ generated by $y$.
Since $y^{l+1}=0$, we have a surjective ring homomorphism $\phi:k[[t]]/(t^{l+1})\to k[[y]]$ given by $\phi(\overline{f(t)})=f(y)$ for $f(t)\in k[[t]]$, where $t$ is an indeterminate over $k$.
Thus we obtain a ring homomorphism
$$
\psi:k[[t]]/(t^{l+1}) \overset{\phi}{\to} k[[y]]\subsetneq R\to R/(x_1,\dots,x_n)+\m^{l+1}=k[[Y]]/(Y^{l+1}).
$$
We see that $\psi$ is an isomorphism.
Hence $\phi$ is injective, and therefore it is an isomorphism.
Applying Corollary \ref{sri} to $S=k[[y]]$ and $I=(x_1,\dots,x_n)+\m^{l+1}$, we get the conclusion.
\end{proof}

Using Corollaries \ref{idealize} and \ref{ctensor}, let us construct examples of a ring $R$ such that $\mod R$ has a nontrivial extension-closed subcategory.

\begin{ex}
Let $k$ be a field.

(1) Consider the ring
$$
R=k[[x,y,z,w]]/(x^2,xy,xz-yw,xw,y^2,yz,z^2,zw,w^2).
$$
This is an Artinian Gorenstein local ring.
Putting $S=k[[x,y]]/(x^2,xy,y^2)$, we observe that $R$ is isomorphic to the idealization $S\ltimes E_S(k)$, where $E_S(k)$ denotes the injective hull of the $S$-module $k$.
Hence it follows from Corollary \ref{idealize} that $\mod R$ has a nontrivial extension-closed subcategory.

In fact, for instance, let $\X$ be the subcategory of $\mod R$ consisting of all $R$-modules $X$ satisfying $\Tor_1^R(R/(x),X)=0$.
It is clear that $\X$ is extension-closed.
We have an exact sequence
$$
0 \to R/(x,y,w) \overset{f}{\to} R \to R/(x) \to 0,
$$
where $f(\overline{1})=x$.
Making the tensor product over $R$ of this exact sequence with $R/(z)$, we get an exact sequence
$$
0 \to \Tor_1^R(R/(x),R/(z)) \to k \overset{g}{\to} R/(z) \to R/(x,z) \to 0,
$$
where $g(\overline{1})=\overline{x}$.
We see that $\Tor_1^R(R/(x),R/(z))=0$, namely, $R/(z)$ belongs to $\X$.
Since $R/(x)$ is not a free $R$-module, $k$ does not belong to $\X$.
Thus $\X$ is an extension-closed subcategory of $\mod R$ which is different from any of $\zero$, $\add R$, $\mod R$.

(2) Let
$$
R=k[[x,y]]/(x^n,y^m)
$$
with $n,m\ge 2$.
This is an Artinian complete intersection.
Since we have an isomorphism $R\cong k[[x]]/(x^n)\operatorname{\widehat{\otimes}}_kk[[y]]/(y^m)$ of rings, $\mod R$ has a nontrivial extension-closed subcategory by Corollary \ref{ctensor}.

Indeed, for example, the subcategory of $\mod R$ consisting of all $R$-modules $X$ with $\Tor_1^R(R/(x),X)=0$ is extension-closed, and does not coincide with any of $\zero$, $\add R$, $\mod R$ because it contains $R/(y)$ and does not contain $k$.
\end{ex}

Now, we verify that Conjecture \ref{mainc} holds for a ring admitting a module with bounded Betti numbers.

\begin{prop}\label{bergh}
Suppose that $\mod R$ has only trivial extension-closed subcategories.
If there exists a nonfree $R$-module $M$ whose Betti numbers are bounded, then $R$ is an Artinian hypersurface.
\end{prop}

\begin{proof}
That the local ring $R$ is Artinian follows from Proposition \ref{artgor}.
Let $\X$ be the subcategory of $\mod R$ consisting of all $R$-modules whose Betti numbers are bounded.
Then it is easy to see that $\X$ is extension-closed.
Since the nonfree $R$-module $M$ belongs to $\X$, our assumption implies that $\X$ coincides with $\mod R$.
In particular, the module $k$ is in $\X$, which forces $R$ to be a hypersurface (cf. \cite{T} or \cite[Remarks 8.1.1(3)]{A}).
\end{proof}

Using \cite[Theorem 3.2]{B}, we observe that such a module $M$ as in Proposition \ref{bergh} exists when there exists an $R$-complex of finite complete intersection dimension and of infinite projective dimension.
(See \cite{AGP} for the details of complete intersection dimension.)
Thus we obtain:

\begin{cor}\label{berghc}
Assume that there exists an $R$-complex of finite complete intersection dimension and of infinite projective dimension.
If $\mod R$ has only trivial extension-closed subcategories, then $R$ is an Artinian hypersurface.
\end{cor}

Since over a complete intersection local ring every module has finite complete intersection dimension, Corollary \ref{berghc} and Proposition \ref{artgor} guarantee that Conjecture \ref{mainc} holds true in the case where the local ring $R$ is a complete intersection.
Combining this with Proposition \ref{mainp}, we get the following result.

\begin{cor}\label{bcor}
If $R$ is a complete intersection, then the following are equivalent.
\begin{enumerate}[\rm (1)]
\item
$R$ is an Artinian hypersurface.
\item
$\mod R$ has only trivial extension-closed subcategories.
\end{enumerate}
\end{cor}

\section{Main results}

In this section, we conduct a closer investigation of the condition that $\mod R$ has a nontrivial extension-closed subcategory.
Establishing a certain assumption on the ring $R$, we shall construct an explicit nontrivial extension-closed subcategory.
For this purpose, we begin with introducing a notion of a subcategory constructed from a single module.

\begin{dfn}
Let $X$ be a nonzero $R$-module.
We define the subcategory $\filt_R^nX$ of $\mod R$ inductively as follows.
\begin{enumerate}[(1)]
\item
Let $\filt_R^1X$ be the subcategory consisting of $X$.
\item
For $n\ge 2$, let $\filt_R^nX$ be the subcategory consisting of all $R$-modules $M$ such that there are exact sequences
$$
0 \to Y \to M \to X \to 0
$$
of $R$-modules with $Y\in\filt_R^{n-1}X$.
\end{enumerate}
We denote by $\filt_RX$ the subcategory of $\mod R$ consisting of all $R$-modules $M$ such that $M\in\filt_R^nX$ for some $n\ge 1$.
\end{dfn}

Here is a result concerning the structure of $\filt_R^nX$.
Its name comes from its property stated in the first assertion.

\begin{prop}\label{fil}
Let $X$ be a nonzero $R$-module.
\begin{enumerate}[\rm (1)]
\item
An $R$-module $M$ belongs to $\filt_R^nX$ if and only if there exists a filtration
$$
0=M_0\subsetneq M_1\subsetneq M_2\subsetneq\cdots\subsetneq M_n=M
$$
of $R$-submodules of $M$ with $M_i/M_{i-1}\cong X$ for all $1\le i\le n$.
\item
If $\filt_R^pX$ intersects $\filt_R^qX$, then $p=q$.
\end{enumerate}
\end{prop}

\begin{proof}
(1) This can be proved by induction on $n$.

(2) It is seen from the definition that if an $R$-module $M$ belongs to $\filt_R^nX$, then we have $\e(M)=n\cdot\e(X)$, where $\e(-)$ denotes the multiplicity.
The assertion immediately follows from this.
\end{proof}

\begin{cor}
Let $X$ be a nonzero $R$-module.
\begin{enumerate}[\rm (1)]
\item
Let $0 \to L \to M \to N \to 0$ be an exact sequence of $R$-modules.
If $L$ is in $\filt_R^pX$ and $N$ is in $\filt_R^qX$, then $M$ is in $\filt_R^{p+q}X$.
\item
The subcategory $\filt_RX$ of $\mod R$ is closed under extensions.
\end{enumerate}
\end{cor}

\begin{proof}
(1) Using Proposition \ref{fil}(1), we can prove the assertion.

(2) This assertion follows from (1).
\end{proof}

For an $R$-module $X$, we denote by $\ext_RX$ the {\em extension closure} of $X$, that  is, the smallest extension-closed subcategory of $\mod R$ containing $X$.
One can describe $\ext_RX$ by using $\filt_RX$.

\begin{prop}\label{extfilt}
Let $X$ be a nonzero $R$-module.
Then $\ext_RX$ coincides with the subcategory of $\mod R$ consisting of all direct summands of modules in $\filt_RX$.
\end{prop}

\begin{proof}
Let $\X$ be the subcategory of $\mod R$ consisting of all direct summands of modules in $\filt_RX$.
It suffices to prove the following two statements.
\begin{enumerate}[(1)]
\item
$\X$ is an extension-closed subcategory of $\mod R$ containing $X$.
\item
If $\X'$ is an extension-closed subcategory of $\mod R$ containing $X$, then $\X'$ contains $\X$.
\end{enumerate}

As to (1):
Obviously, $\X$ contains $X$ and is closed under direct summands.
Let $0 \to L \to M \to N \to 0$ be an exact sequence of $R$-modules with $L,N\in\X$.
Then we have isomorphisms $L\oplus L'\cong Y$ and $N\oplus N'\cong Z$ for some $L',N'\in\mod R$ and $Y,Z\in\filt X$.
Taking the direct sum of the above exact sequence with the exact sequences $0 \to L' \overset{=}{\to} L' \to 0 \to 0$ and $0 \to 0 \to N' \overset{=}{\to} N' \to 0$, we get an exact sequence
$$
0 \to Y \to L'\oplus M\oplus N' \to Z \to 0.
$$
Since $Y,Z$ are in $\filt X$, so is $L'\oplus M\oplus N'$, and hence $M$ belongs to $\X$.
Thus $\X$ is closed under extensions.

As to (2):
Since $\X'$ is closed under direct summands, we have only to prove that $\X'$ contains $\filt X$, equivalently, that $\X'$ contains $\filt^nX$ for every $n\ge 1$.
This can easily be shown by induction on $n$.
\end{proof}

Let $x$ be an element of $R$.
To understand the subcategory $\ext_R(R/(x))$, we investigate the structure of each module in $\filt_R^n(R/(x))$ for $n\ge 1$.

\begin{prop}\label{matrix}
Let $x\in R$ and $n\ge 1$.
Let $M$ be an $R$-module in $\filt_R^n(R/(x))$.
Then there exists an exact sequence
$$
\begin{CD}
R^n @>{
\begin{pmatrix}
x & c_{1,2} & \cdots & c_{1,n} \\
0 & \ddots & \ddots & \vdots \\
\vdots & \ddots & \ddots & c_{n-1,n} \\
0 & \cdots & 0 & x
\end{pmatrix}
}>> R^n @>>> M @>>> 0
\end{CD}
$$
of $R$-modules with each $c_{i,j}$ being in $R$ such that
$$
\begin{pmatrix}
c_{1,j} \\
\vdots \\
c_{j-1,j}
\end{pmatrix}
(0:_Rx)\subseteq\Image
\begin{pmatrix}
x & c_{1,2} & \cdots & c_{1,j-1} \\
0 & \ddots & \ddots & \vdots \\
\vdots & \ddots & \ddots & c_{j-2,j-1} \\
0 & \cdots & 0 & x
\end{pmatrix}
$$
for all $2\le j\le n$.
\end{prop}

\begin{proof}
We prove the proposition by induction on $n$.
When $n=1$, we have $M\cong R/(x)$, and there is an exact sequence
$R \overset{x}{\to} R \to M \to 0$.
Let $n\ge 2$.
We have an exact sequence $0 \to Y \to M \to R/(x) \to 0$ of $R$-modules with $Y\in\filt^{n-1}(R/(x))$.
The induction hypothesis shows that there is an exact sequence $R^{n-1} \overset{A}{\to} R^{n-1} \to Y \to 0$ with $A=\left(
\begin{smallmatrix}
x & c_{1,2} & \cdots & c_{1,n-1} \\
0 & \ddots & \ddots & \vdots \\
\vdots & \ddots & \ddots & c_{n-2,n-1} \\
0 & \cdots & 0 & x
\end{smallmatrix}
\right)$ such that $\left(
\begin{smallmatrix}
c_{1,j} \\
\vdots \\
c_{j-1,j}
\end{smallmatrix}
\right)(0:x)\subseteq\Image\left(
\begin{smallmatrix}
x & c_{1,2} & \cdots & c_{1,j-1} \\
0 & \ddots & \ddots & \vdots \\
\vdots & \ddots & \ddots & c_{j-2,j-1} \\
0 & \cdots & 0 & x
\end{smallmatrix}
\right)$ for all $2\le j\le n-1$.
We have a commutative diagram
$$
\begin{CD}
@. @. @. 0 \\
@. @. @. @VVV \\
@. @. @. (0:x) \\
@. @. @. @VVV \\
0 @>>> R^{n-1} @>{\left(
\begin{smallmatrix}
1 \\
0
\end{smallmatrix}
\right)}>> R^{n-1}\oplus R @>{\left(
\begin{smallmatrix}
0 & 1
\end{smallmatrix}
\right)}>> R @>>> 0 \\
@. @V{A}VV @V{\left(
\begin{smallmatrix}
A & B \\
0 & x
\end{smallmatrix}
\right)}VV @V{x}VV \\
0 @>>> R^{n-1} @>{\left(
\begin{smallmatrix}
1 \\
0
\end{smallmatrix}
\right)}>> R^{n-1}\oplus R @>{\left(
\begin{smallmatrix}
0 & 1
\end{smallmatrix}
\right)}>> R @>>> 0 \\
@. @V{f}VV @VVV @VVV \\
0 @>>> Y @>>> M @>>> R/(x) @>>> 0 \\
@. @VVV @VVV @VVV \\
@. 0 @. 0 @. 0
\end{CD}
$$
with exact rows and columns.
The induced map $g:(0:x)\to Y$ is the zero map by the snake lemma.
By diagram chasing, we see that $g(r)=f(Br)$ holds for each $r\in(0:x)$.
Hence we have $f(Br)=0$ for all $r\in(0:x)$, whence $Br$ is in the image of the map $A:R^{n-1}\to R^{n-1}$.
Writing $B=\left(
\begin{smallmatrix}
c_{1,n} \\
\vdots \\
c_{n-1,n}
\end{smallmatrix}
\right)$, we obtain an inclusion relation $\left(
\begin{smallmatrix}
c_{1,n} \\
\vdots \\
c_{n-1,n}
\end{smallmatrix}
\right)(0:x)\subseteq\Image\left(
\begin{smallmatrix}
x & c_{1,2} & \cdots & c_{1,n-1} \\
0 & \ddots & \ddots & \vdots \\
\vdots & \ddots & \ddots & c_{n-2,n-1} \\
0 & \cdots & 0 & x
\end{smallmatrix}
\right)$.
Consequently, we have
$$
\left(
\begin{smallmatrix}
c_{1,j} \\
\vdots \\
c_{j-1,j}
\end{smallmatrix}
\right)(0:x)\subseteq\Image\left(
\begin{smallmatrix}
x & c_{1,2} & \cdots & c_{1,j-1} \\
0 & \ddots & \ddots & \vdots \\
\vdots & \ddots & \ddots & c_{j-2,j-1} \\
0 & \cdots & 0 & x
\end{smallmatrix}
\right)
$$
for all $2\le j\le n$.
The middle column of the above diagram gives an exact sequence
$$
\begin{CD}
R^n @>{\left(
\begin{smallmatrix}
x & c_{1,2} & \cdots & c_{1,n-1} & c_{1,n} \\
0 & \ddots & \ddots & \vdots & \vdots \\
\vdots & \ddots & \ddots & c_{n-2,n-1} & c_{n-2,n} \\
0 & \cdots & 0 & x & c_{n-1,n} \\
0 & \cdots & 0 & 0 & x
\end{smallmatrix}
\right)}>> R^n @>>> M @>>> 0.
\end{CD}
$$
Thus the proof of the proposition is completed.
\end{proof}

Now we can prove the following result concerning the structure of $\ext_R(R/(x))$, which is the main result of this paper.

\begin{thm}\label{main}
Let $x,y$ be part of a minimal system of generators of $\m$ with $xy=0$.
Then $k$ does not belong to $\ext_R(R/(x))$.
\end{thm}

\begin{proof}
Let $e$ be the embedding dimension of $R$.
We have $e\ge 2$, and write $\m=(x,y,z_3,\dots,z_e)$.
Let us assume that $k$ belongs to $\ext_R(R/(x))$.
We want to derive a contradiction.
By Proposition \ref{extfilt}, the module $k$ is isomorphic to a direct summand of a module $M\in\filt_R(R/(x))$.
We have an isomorphism $M\cong k\oplus N$ for some $R$-module $N$, and $M$ belongs to $\filt_R^n(R/(x))$ for some $n\ge 1$.
Proposition \ref{matrix} gives an exact sequence
\begin{equation}\label{moto}
\begin{CD}
R^n @>{\left(
\begin{smallmatrix}
x & c_{1,2} & \cdots & c_{1,n} \\
0 & \ddots & \ddots & \vdots \\
\vdots & \ddots & \ddots & c_{n-1,n} \\
0 & \cdots & 0 & x
\end{smallmatrix}
\right)}>> R^n @>>> M @>>> 0
\end{CD}
\end{equation}
of $R$-modules such that
$$
\left(
\begin{smallmatrix}
c_{1,j} \\
\vdots \\
c_{j-1,j}
\end{smallmatrix}
\right)(0:x)\subseteq\Image\left(
\begin{smallmatrix}
x & c_{1,2} & \cdots & c_{1,j-1} \\
0 & \ddots & \ddots & \vdots \\
\vdots & \ddots & \ddots & c_{j-2,j-1} \\
0 & \cdots & 0 & x
\end{smallmatrix}
\right)
$$
for all $2\le j\le n$.
Since $y$ is in $(0:x)$, there are elements $d_{1,j},\dots,d_{j-1,j}\in R$ such that
$$
\left(
\begin{smallmatrix}
c_{1,j}y \\
\vdots \\
c_{j-1,j}y
\end{smallmatrix}
\right)=\left(
\begin{smallmatrix}
x & c_{1,2} & \cdots & c_{1,j-1} \\
0 & \ddots & \ddots & \vdots \\
\vdots & \ddots & \ddots & c_{j-2,j-1} \\
0 & \cdots & 0 & x
\end{smallmatrix}
\right)\left(
\begin{smallmatrix}
d_{1,j} \\
\vdots \\
d_{j-1,j}
\end{smallmatrix}
\right).
$$
Hence the equality
$$
c_{i,j}y=xd_{i,j}+c_{i,i+1}d_{i+1,j}+\cdots+c_{i,j-1}d_{j-1,j}
$$
holds for $2\le j\le n$ and $1\le i\le j-1$.

We claim that the elements $c_{i,j},d_{i,j}$ belong to $\m$ for all $2\le j\le n$ and $1\le i\le j-1$.
Indeed, the hypothesis of induction on $j$ implies that $c_{i,l}$ is in $\m$ for $i+1\le l\le j-1$, and the assumption of descending induction on $i$ shows that $d_{l,j}$ is in $\m$ for $i+1\le l\le j-1$.
Hence we have $c_{i,j}y-xd_{i,j}\in\m^2$, which gives an equality
$$
\overline{c_{i,j}}\cdot\overline{y}-\overline{x}\cdot\overline{d_{i,j}}=\overline{0}
$$
in $\m/\m^2$.
Since $\overline{x},\overline{y}$ are part of a $k$-basis of $\m/\m^2$, we have $\overline{c_{i,j}}=\overline{d_{i,j}}=\overline{0}$ in $k$.
Therefore, $c_{i,j},d_{i,j}$ belong to $\m$, as desired.

By elementary column operations, the matrix $\left(
\begin{smallmatrix}
x & c_{1,2} & \cdots & c_{1,n} \\
0 & \ddots & \ddots & \vdots \\
\vdots & \ddots & \ddots & c_{n-1,n} \\
0 & \cdots & 0 & x
\end{smallmatrix}
\right)$ can be transformed into a matrix $\left(
\begin{smallmatrix}
x & b_{1,2} & \cdots & b_{1,n} \\
0 & \ddots & \ddots & \vdots \\
\vdots & \ddots & \ddots & b_{n-1,n} \\
0 & \cdots & 0 & x
\end{smallmatrix}
\right)$ such that each $b_{i,j}$ is an element of the ideal $I=(y,z_3,\dots,z_e)$.
We have an exact sequence
$$
\begin{CD}
R^n @>{\left(
\begin{smallmatrix}
x & b_{1,2} & \cdots & b_{1,n} \\
0 & \ddots & \ddots & \vdots \\
\vdots & \ddots & \ddots & b_{n-1,n} \\
0 & \cdots & 0 & x
\end{smallmatrix}
\right)}>> R^n @>>> M @>>> 0,
\end{CD}
$$
and applying $-\otimes_RR/I$ to this, we get an exact sequence
$$
\begin{CD}
(R/I)^n @>{\left(
\begin{smallmatrix}
x & 0 & \cdots & 0 \\
0 & \ddots & \ddots & \vdots \\
\vdots & \ddots & \ddots & 0 \\
0 & \cdots & 0 & x
\end{smallmatrix}
\right)}>> (R/I)^n @>>> M/IM @>>> 0.
\end{CD}
$$
Hence we have an isomorphism $M/IM\cong (R/I+(x))^n=k^n$.
Since $M/IM\cong k\oplus N/IN$, we see that $N/IN$ is isomorphic to $k^{n-1}$, and get an equality
\begin{equation}\label{a}
\beta_1^{R/I}(N/IN)=(n-1)\beta_1^{R/I}(k)
\end{equation}
of Betti numbers.
There is an exact sequence $R^{\beta_1^R(N)} \to R^{\beta_0^R(N)} \to N \to 0$ of $R$-modules, and tensoring $R/I$ with this gives an exact sequence $(R/I)^{\beta_1^R(N)} \to (R/I)^{\beta_0^R(N)} \to N/IN \to 0$ of $R/I$-modules.
It follows from this that
\begin{equation}\label{i}
\beta_1^{R/I}(N/IN)\le\beta_1^R(N).
\end{equation}
The isomorphism $M\cong k\oplus N$ shows
\begin{equation}\label{u}
\beta_1^R(M)=\beta_1^R(k)+\beta_1^R(N)=e+\beta_1^R(N).
\end{equation}
The existence of the exact sequence \eqref{moto} implies
\begin{equation}\label{e}
\beta_1^R(M)\le n.
\end{equation}
Since $\m/I=x(R/I)$ and $x\notin I$, we have
\begin{equation}\label{o}
\beta_1^{R/I}(k)=1.
\end{equation}
Using the (in)equalities \eqref{a}--\eqref{o}, we obtain
$$
n-1=(n-1)\beta_1^{R/I}(k)=\beta_1^{R/I}(N/IN)\le\beta_1^R(N)=\beta_1^R(M)-e\le n-e,
$$
whence $e\le 1$.
This is a desired contradiction; this contradiction completes the proof of the theorem.
\end{proof}

Let $R$ be an Artinian local ring.
Then, using the fact that every $R$-module $M$ is annihilated by the ideal $\m^{\ell(M)}$, we can check that the equality $\m^{\ell(R)-\edim R+1}=0$ holds.
(Here, $\ell(M)$ and $\edim R$ denote the length of $M$ and the embedding dimension of $R$, respectively.)
Recall that $R$ is called {\em stretched} if $\m^i\ne 0$ for all $i<\ell(R)-\edim R+1$, or equivalently, if $\m^{\ell(R)-\edim R}\ne 0$.

\begin{ex}
(1) Every Artinian Gorenstein local ring $R$ with $\m^3=0$ that is not a field is stretched.

(2) Let $k$ be a field, and let
$$
R=k[[x,y,z]]/(xy,xz,yz,x^3-y^2,x^3-z^2)
$$
be a residue ring of a formal power series ring over $k$.
Then $R$ is an Artinian Gorenstein local ring.
Since $\ell(R)=6$, $\edim R=3$ and $\m^3=(x^3)\ne 0$, the ring $R$ is stretched.
\end{ex}

Now we have a sufficient condition for $\mod R$ to have a nontrivial extension-closed subcategory.

\begin{cor}\label{sagnotes}
Let $R$ be a stretched Artinian Gorenstein local ring with $\edim R\ge 2$.
Then $\mod R$ has a nontrivial extension-closed subcategory.
\end{cor}

\begin{proof}
If $\edim R<\ell(R)-2$, then by \cite[Theorem 1.1]{S} there exist elements $x,y\in R$ with $xy=0$ which form part of minimal system of generators of $\m$, and Theorem \ref{main} shows that $\ext_R(R/(x))$ is a nontrivial extension-closed subcategory of $\mod R$.

Let $\edim R\ge\ell(R)-2$.
Then we have $\m^3=0$.
Take an element $x\in\m\setminus\m^2$.

First, assume that $(0:x)$ is not contained in $(x)+\m^2$.
Then there exists an element $y\in (0:x)$ which does not belong to $(x)+\m^2$, and we see that $\overline{x},\overline{y}$ form part of a $k$-basis of $\m/\m^2$.
Hence $x,y$ are part of a minimal system of generators of $\m$ with $xy=0$, and the assertion follows from Theorem \ref{main}.

Next, assume that $(0:x)$ is contained in $(x)+\m^2$.
Then we have
$$
(x)\overset{\rm (a)}{=}(0:(0:x))\supseteq(0:(x)+\m^2)=(0:x)\cap(0:\m^2)\overset{\rm (b)}{=}(0:x).
$$
Here, the equality (a) follows from the double annihilator property (cf. \cite[Exercise 3.2.15]{BH}), and (b) from the inclusion $(0:\m^2)\supseteq\m$.
Suppose that $(0:x)\ne(x)$.
Then we have $x\m\subseteq\m^2\subseteq(0:x)\subsetneq(x)$ and $\ell_R((x)/x\m)=1$, which imply $x\m=\m^2=(0:x)$.
Hence $\m\subseteq(0:\m^2)=(0:(0:x))=(x)$, which contradicts the assumption that $\edim R\ge 2$.
Thus the equality $(0:x)=(x)$ holds, and there exists an exact sequence
$$
\cdots \overset{x}{\to} R \overset{x}{\to} R \overset{x}{\to} R \to R/(x) \to 0
$$
of $R$-modules.
This implies that $R/(x)$ belongs to the subcategory $\X$ of $\mod R$ consisting of all $R$-modules with bounded Betti numbers, which is extension-closed.
Hence $\X$ is neither $\zero$ nor $\add R$, and we also have $\X\ne\mod R$ because $R$ is not a hypersurface by the assumption that $\edim R\ge 2$ again.
Therefore $\X$ is a nontrivial extension-closed subcategory of $\mod R$.
\end{proof}

We can guarantee that our Conjecture \ref{mainc} holds true for a stretched Artinian Gorenstein local ring.
The following result follows from Proposition \ref{mainp} and Corollary \ref{sagnotes}.

\begin{cor}\label{strcor}
Let $R$ be a stretched Artinian Gorenstein local ring.
Then the following are equivalent.
\begin{enumerate}[\rm (1)]
\item
$R$ is an Artinian hypersurface.
\item
$\mod R$ has only trivial extension-closed subcategories.
\end{enumerate}
\end{cor}

We end this paper by posing a question.

\begin{ques}
An extension-closed subcategory of $\mod R$ is called {\em resolving} if it contains $R$ and is closed under syzygies.
Does the assumption of Theorem \ref{main} imply that $k$ does not belong to the smallest resolving subcategory of $\mod R$ containing $R/(x)$?
\end{ques}

\section*{Acknowledgments}

The author is indebted to Shiro Goto, Petter Andreas Bergh and Kei-ichiro Iima for their valuable comments and useful suggestions.


\end{document}